\documentclass[12pt, oneside,reqno]{amsart}
\usepackage{hyperref}
\usepackage{geometry}
\usepackage{amsmath}
\usepackage{cite}
\usepackage{xcolor}
\usepackage{graphicx}
\usepackage{subfigure}

\newcommand{\rn}{\mathbb R^n}
\newcommand{\sn}{S^{n-1}}
\newcommand{\kno}{\mathcal K^n_o}
\newcommand{\kne}{\mathcal K^n_e}

\newcommand\wtilde[1]{\overset{\lower.4ex\hbox{$\scriptstyle \sim$}}{#1}}

\numberwithin{equation}{section}

\newtheorem{thm}{Theorem}[section]

\newtheorem{lem}[thm]{Lemma}
\newtheorem{pro}[thm]{Problem}

\begin{document}
	\title[On the dual quermassintegral]{$L_p$ Minkowski problem and  Brunn-Minkowski inequality for dual quermassintegrals}

	\author[X. Chen]{Xiaojuan Chen}
	\address{Institute of Mathematics,
		Hunan University,  Changsha,  410082,  China}
	\email{
		cxj@hnu.edu.cn}
	
	\author[S. Tang]{Shengyu Tang}
	\address{Institute of Mathematics,
		Hunan University,  Changsha,  410082,  China}
	\email{
		tsy@hnu.edu.cn}
	
	\author[S. Wang]{Sinan Wang}
	\address{Institute of Mathematics,
		Hunan University,  Changsha,  410082,  China}
	\email{
		wangsinan@hnu.edu.cn}

	
	\thanks{Research of Chen was supported by the Postdoctoral Fellowship Program of CPSF under Grant Number GZB20250702.}
	
	\begin{abstract}
		This paper studies the core problems in the $L_p$ dual Brunn-Minkowski theory, encompassing the $L_p$ Minkowski problem and $L_p$ Brunn-Minkowski inequality for dual quermassintegrals. For the case $0<p<q\leq n$, we establish $C^0$ estimates for the $L_p$ dual Minkowski problem without symmetric assumptions, thereby resolving a related problem proposed by B\"or\"oczky-Chen-Liu-Saroglou in the smooth sense. We further prove the uniqueness of smooth solutions under appropriate conditions, provided the density function is sufficiently close to a constant in the H\"older norm. 
		Finally, exploiting the fact that the uniqueness of the Minkowski type problem is equivalent to the validity of the Brunn-Minkowski inequality in a certain sense, we study the $L_p$ Brunn-Minkowski inequality for dual quermassintegrals for origin-symmetric convex bodies with $p<q$.
	\end{abstract}
	\maketitle	
	
	\section{Introduction}
	Convex geometric analysis, the analytically oriented branch of convex geometry, is built in part on the classical Brunn-Minkowski (or mixed volume) theory developed by Minkowski, Alexandrov, and other pioneers. In this classical framework, the quermassintegrals are core geometric invariants, which correspond to integrals of mean curvature for closed convex hypersurfaces in differential geometry, as well as integrals of projection areas in integral geometry.
	Viewed as functionals on the space of convex bodies in $\rn$, the derivatives of the quermassintegrals yield surface area measures, whose characterization is a foundational problem in the field, epitomized by the Minkowski problem that links convex geometry to fully nonlinear partial differential equations. Here, tools such as the cosine transform (a variant of the Fourier transform) and elliptic partial differential equations on the unit sphere play crucial roles in the study of projections and surface area measures. Motivated by research on convex body sections in integral geometry, an analogous theory for mixed volumes was introduced in the 1970s by Lutwak \cite{L75}, which revealed a striking duality in convex geometry and thus became known as the dual mixed volume theory, or the dual Brunn-Minkowski theory.
	The central invariants of this dual theory are the dual quermassintegrals, special dual mixed volumes where volume serves as both the classical and dual quermassintegrals.
	
	\subsection{Minkowski problem}
	The classical Minkowski problem is essentially the characterization of surface area measures, which was first studied by Minkowski for polytopes in 1897. In 1962, Firey \cite{F62} generalized the Minkowski sum to $p$-sum. Later, Lutwak \cite{L93} developed the corresponding $L_p$ Brunn-Minkowski theory based on Firey's $p$-sum, and proposed $L_p$ surface area measures, then utilized Aleksandrov's variational theory to solve the corresponding $L_p$ Minkowski problem.
	The $L_p$ Minkowski problem encompassed many important issues. When $p = 1$, it corresponds to the classical Minkowski problem \cite{A38, M97}; when $p = 0$, it is called the logarithmic Minkowski problem \cite{BLYZ13, Z14}; when $p =-n$, it is known as the centro- affine Minkowski problem \cite{CW06, Z15}.
	
	The above Minkowski problems were proposed regarding the surface area measure. In convex geometry, there is another important measure, namely the curvature measure. The most typical curvature measure is the Aleksandrov integral curvature measure. Aleksandrov \cite{A42} obtained the existence of the corresponding solution to the Aleksandrov problem. In 2016, a breakthrough was made when Huang-Lutwak-Yang-Zhang \cite{HLYZ16} discovered the $q$-th dual curvature measure $\tilde{C}_q(K,\cdot)$, which can be
	viewed as differentials of dual quermassintegrals, and derived the related variational formulas. Then they proposed and solved the corresponding dual Minkowski problem. 
	Subsequently, the dual Brunn-Minkowski theory flourished and achieved a series of significant results, see \cite{BLYZZ19, HP18, H25, Hu25, LSW20, Z17, Z18} and the references therein. 
	
	Lutwak-Yang-Zhang \cite{LYZ18} further integrated the dual Brunn-Minkowski theory into the $L_p$ Brunn-Minkowski theory, and introduced the $L_p$ dual curvature measure $\tilde{C}_{p,q}(K,\cdot)$ to unify the $L_p$ surface area measure and the $q$-th dual curvature measure. The associated Minkowski problem was called the $L_p$ dual Minkowski problem. Subsequently, many researchers adopted variational methods to derive solutions with variational structure:
	B\"{o}r\"{o}czky-Fodor \cite{BF19} obtained the existence when $p>1, q>0, p\neq q$; Chen-Chen-Li \cite{CCL21} studied the existence and non-uniqueness of the solution under even assumptions when $q>0, p<0$ in some certain assumption; Li-Liu-Lu \cite{LLL22} constructed counterexamples to demonstrate the nonuniqueness when $p<0<q$; Mui \cite{M24} studied the existence of the solution under even assumptions when $-1<p<0, q<1+p, p\neq q$, and was further extended by B\"{o}r\"{o}czky-Kov\'{a}cs-Mui-Zhang \cite{BKMZ25} under group symmetric assumptions. After that, Shan \cite{S25} further established the existence of solutions from an algebraic perspective for group invariant measures and group invariant convex bodies when $p,q\in \mathbb{R}$.
	
	When the $L_p$ dual curvature measure has a density $\frac{d\tilde{C}_{p,q}(K,\cdot)}{d\mathcal{H}^{n-1}}=f$ with respect to spherical Lebesgue measure $\mathcal{H}^{n-1}$, then this problem is equivalent to solving the following Monge‑Amp\`ere equation
	\begin{equation}
		\label{MA equation in introduction}
		h^{1-p}(h^2+|\nabla h|^2)^{\frac{q-n}{2}}\det(\nabla^2h+hI)=f\quad \text{on}~\sn,
	\end{equation}
	where $\nabla h$ and $\nabla^2 h$ denote the spherical gradient and the Hessian of the function $h$ on $S^{n-1}$ with respect to a moving orthonormal frame respectively, and $I$ is the identity matrix.
	For equation \eqref{MA equation in introduction}, Huang-Zhao \cite{HZ18} used the continuity method to obtain the existence and the uniqueness of smooth solutions when $p>q$; Chen-Huang-Zhao \cite{CHZ19} used the curvature flow method to obtain the existence of solutions when $pq > 0$ and $f$ is even; Chen-Li \cite{CL21} used the curvature flow method to obtain the existence of solutions when $p>0, q\neq n$ or $q \leq p<0$. 
	For a special case of the uniqueness results when $f\equiv1$, which is called the isotropic case, have also attracted attention of many scholars. In contrast to \cite{HZ18} which solved the case $p>q$ for general $f$, Chen–Huang–Zhao \cite{CHZ19} solved the case $-n\leq p<q\leq\min\{n,n+p\}$ under the assumption of origin-symmetric. Chen-Li \cite{CL21} removed the symmetric assumption for the case $1<p<q\leq n$ or $-n\leq p<q<-1$ or $p=q$ up to rescaling. Hu-Ivaki \cite{HI24} gave the uniqueness for the case $-n<p\leq -1,n\leq q\leq n+1$.
	For the uniqueness of a large generalized class of isotropic curvature problems, see Ivaki-Milman \cite{IM23} and Li-Wan \cite{LW24}. Li-Wan \cite{LW22} and Liu-Lu \cite{LL26} classified the solutions for the planar case.
	
	In this paper, we mainly focus on the uniqueness of the $L_p$ dual Minkowski problem when the density function is near the ball in $C^{\alpha}$ norm sense. The uniqueness of this viewpoint has become a hot research topic recently. Chen-Huang-Li-Liu \cite{CHLL20} established the uniqueness of $L_p$ Minkowski problem when $p\in\left(1-\frac{c}{n^{\frac{3}{2}}},1\right)$ and the density function is even and close to the ball. Chen-Feng-Liu \cite{CFL22} further studied the uniqueness for logarithmic Minkowski problem when $n=3$ without any symmetric assumptions. For higher dimensions, the uniqueness results were independently established by B\"{o}r\"{o}czky-Saroglou \cite{BS24} and Hu-Ivaki \cite{HI25} using different approaches. 
	Hu \cite{Hu25} established the uniqueness for symmetric solutions of the dual Minkowski problem when the density function is even, close to the ball, and $0<q\leq n$ for $1\leq n\leq 3$ or $n-3\leq q\leq n$ for $n>3$. Cabezas-Moreno and Hu \cite{CH25} proved the uniqueness for symmetric solutions of the $L_p$ dual Minkowski problem when $1<p<q\leq n$, which also developed by B\"{o}r\"{o}czky-Chen-Liu-Saroglou \cite{BCLS-1} for $p>-1,q>0,p<q<\min\{n,n+p\}$ with symmetric assumptions and for $p\in (-1,1)$, $q$ is close to $n$ without symmetric assumptions.
	
	It is worth noting that $C^0$ estimates are important to derive existence and uniqueness results for the above Minkowski problems and the corresponding Monge-Amp\`{e}re equations. Recently, 
	B\"{o}r\"{o}czky-Chen-Liu-Saroglou \cite{BCLS25-2} proposed the following problem for the $L_p$ dual Minkowski problem.
	\begin{pro}[BCLS \cite{BCLS25-2}]
		\label{problem}
		For $n\geq2$, $p\in(-1,1),q>n-1$ and $\lambda>1$, does there exist a constant $C=C(n,p,q,\lambda)>1$ such that if
		\begin{equation*}
			\lambda^{-1}\mathcal{H}^{n-1}\leq\tilde{C}_{p,q}(K,\cdot)\leq \lambda\mathcal{H}^{n-1}
		\end{equation*}
		holds on $S^{n-1}$ for a convex body $K\in\kno$, then the support function $h_K$ of $K$ satisfies
		\begin{equation*}
			\max_{x\in\sn}h_{K}(x)\leq C\ \text{and}\ |K|\geq\frac{1}{C}.
		\end{equation*}
	\end{pro}
	In \cite{BCLS25-2}, they solved Problem \ref{problem} for $p\in[0,1),q>p+2$ in $\mathbb{R}^3$.
	They also verified Problem \ref{problem} in \cite{BCLS-1} when $n=3,4$ and $q$ is very close to $n$, and higher dimension for $-1<p<q<\min\{n,n+p\}$ with $q>0$ under symmetric assumptions. In our paper, we first answer Problem \ref{problem} when $q\in(1,n]$, $p\in(0,q)$ and the given data is smooth without symmetric assumptions.
	\begin{thm}
		\label{C0 estimate in introduction}
		For $n\geq2$, $q\in(1,n]$ and $p\in(0,q)$, let $f$ be a smooth function on $\sn$.  There exists a constant $C$ depending on $n,p,q,\lambda$ such that if
		\begin{equation*}
			\lambda^{-1}\mathcal{H}^{n-1}\leq\tilde{C}_{p,q}(K,\cdot)\leq \lambda\mathcal{H}^{n-1}
		\end{equation*}
		with $d\tilde{C}_{p,q}(K,\cdot)=fd\mathcal{H}^{n-1}$ for any $K\in\kno$, then
		\begin{equation*}
			\max_{x\in\sn}h_{K}(x)\leq C\ \text{and}\ |K|\geq\frac{1}{C}.
		\end{equation*}
	\end{thm}
	Using the similar argument of B\"{o}r\"{o}czky-Chen-Liu-Saroglou \cite{BCLS-1}, we deduce the following uniqueness results. Along the way, we also use the uniqueness results of \eqref{MA equation in introduction} in Chen-Li \cite{CL21} when $f\equiv1$ and $1<p<q\leq n$.
	\begin{thm}\label{uniqueness result}
		Suppose $n\geq2$, $\alpha\in(0,1)$, $q\in(1,n]$, $p\in(0,q)$. There exists a constant $\epsilon>0$ depending only on $n,\alpha,p$ and $q$ such that, if $f$ is a smooth function on $\sn$ satisfying $\|f-1\|_{C^\alpha(\sn)}<\epsilon$, and
		\begin{itemize}
			\item either, $p\in(0,1)$ and $q\in(n-\epsilon,n]$,
			\item or $1<p<q\leq n$,
		\end{itemize}
		then equation \eqref{MA equation in introduction}
		has a unique smooth solution in $\kno$.
	\end{thm}
	
	\subsection{Brunn-Minkowski inequality}
	The Brunn-Minkowski inequality, as another important topic in Brunn-Minkowski theory, states that for given convex bodies $K$ and $L$ in $\mathbb{R}^n$,
	\begin{equation*}
		V_n((1-\lambda)K+\lambda L)^{\frac{1}{n}}\geq (1-\lambda)V_n(K)^{\frac{1}{n}}+\lambda V_n(L)^{\frac{1}{n}},\quad \forall~\lambda\in[0,1],
	\end{equation*}
	with equality if and only if $K$ and $L$ are homothetic. This is equivalent to the log-concavity of the volume functional 
	\begin{equation*}
		V_n((1-\lambda)K+\lambda L)\geq V_n(K)^{1-\lambda}V_n(L)^{\lambda},\quad \forall~\lambda\in[0,1].
	\end{equation*}
	The relationship between the Brunn-Minkowski inequality and other inequalities in geometry and analysis, along with some applications, refer to the survey article by Gardner \cite{G02}. 
	
	In 1962, Firey \cite{F62} generalized the Minkowski sum of convex bodies to the $L_p$-sum $(1-\lambda)K+_p\lambda L$.
	When $p\geq1$, it is defined as the convex body with support function $((1-\lambda)h_K^p+\lambda h_L^p)^{\frac{1}{p}}$. Firey also established the corresponding $L_p$ Brunn-Minkowski inequality
	\begin{equation}\label{E:10}
		V_n((1-\lambda)K+_p\lambda L)^{\frac{p}{n}}\geq (1-\lambda)V_n(K)^{\frac{p}{n}}+\lambda V_n(L)^{\frac{p}{n}}, \quad\forall~\lambda\in[0,1],
	\end{equation}
	with equality if and only if $K=L$. 
	
	A breakthrough came from B\"or\"oczky-Lutwak-Yang-Zhang \cite{BLYZ12}, which studied the $L_p$ Brunn-Minkowski inequality for $p\in[0,1)$, covering the previously missing range. It is well-known that $((1-\lambda)h_K^p+\lambda h_L^p)^{\frac{1}{p}}$ cannot be a support function for any convex body when $p\in[0,1)$. However, they found a natural generalization of $L_p$-sum as follows
	\begin{equation*}
		(1-\lambda)K+_p\lambda L
		=\cap_{x\in\mathbb{S}^{n-1}}\{z\in\mathbb{R}^n: x\cdot z\leq ((1-\lambda)h_K^p(x)+\lambda h_L^p(x))^{\frac{1}{p}}\}.
	\end{equation*}
	When $p=0$, it is actually the log Minkowski sum,
	\begin{equation*}
		(1-\lambda)K+_0\lambda L
		=\cap_{x\in\mathbb{S}^{n-1}}\{z\in\mathbb{R}^n: x\cdot z\leq h_K(x)^{1-\lambda}h_L(x)^{\lambda}\},
	\end{equation*}
	where $K$ and $L$ are convex bodies that contain the origin in their interiors. Meanwhile, when $p\in(0,1)$, for two origin-symmetric convex bodies $K$ and $L$ in $\mathbb{R}^n$, they conjectured the $L_p$ Brunn-Minkowski inequality \eqref{E:10}
	and the log Brunn-Minkowski inequality when $p=0$,
	\begin{equation*}
		V_n((1-\lambda)K+_0\lambda L)\geq V_n(K)^{1-\lambda}V_n(L)^{\lambda} \quad\forall~\lambda\in[0,1].
	\end{equation*}

	Recently, Kolesnikov-Milman \cite{KM22} introduced a local perspective: instead of proving the inequality for all convex bodies directly, they studied its infinitesimal version. Hence, they developed a local $L_p$ Brunn-Minkowski inequality for $p\in[1-\frac{c}{n^{\frac{3}{2}}},1)$, which stated that when $K, L\in \mathcal{K}_{+,e}^2$ satisfies
	\begin{equation*}
		(1-\lambda)K+_p\lambda L\in\mathcal{K}_{+,e}^2,\quad\forall~\lambda\in[0,1],
	\end{equation*}
	where $\mathcal{K}_{+,e}^2$ is the class of origin-symmetric convex bodies with $C^2$ smooth boundary and strictly positive curvature, then \eqref{E:10} holds for $n\geq 2$.
	Moreover,
	Chen-Huang-Li-Liu \cite{CHLL20} gave the equivalence of the $L_p$ Brunn-Minkowski inequality \eqref{E:10} to the local version, by adapting the PDE methods. This result showed that studying the local inequality is not merely a technical simplification but a legitimate route to the global conjecture. Putterman \cite{P21} also proved the equivalence between local and global from a geometric perspective. 
	
	Corresponding to the near spherical uniqueness of Minkowski problem discussed in the preceding context, we further investigate the Brunn-Minkowski inequality in the near spherical setting. The study of near spherical Brunn-Minkowski inequality can be traced back to the work in Colesanti-Livshyts-Marsiglietti \cite{CLM17}. They investigated the Brunn-Minkowski inequality (and log version) for radially symmetric log-concave measures when the given two convex bodies are near the ball. Bianchini-Colesanti-Pagnini-Roncoroni \cite{BCPR23} verified the validity of the $L_p$ Brunn–Minkowski inequality for quermassintegrals of symmetric convex bodies in $\mathbb{R}^n$, under the condition that one body is the ball and the other one lies in a small neighborhood of the ball. Subsequently, Patsalos-Saroglou \cite{PS26} improved the results in \cite{BCPR23} to a symmetric convex body near the ball, while the other one is arbitrary symmetric convex body. Furthermore, they proved the near spherical uniqueness of the corresponding Minkowski problem based on the near spherical Brunn-Minkowski inequality.
	
	The Brunn-Minkowski inequality for dual quermassintegrals was conjectured by Lutwak and has been presented in various talks and seminars by the Lutwak-Yang-Zhang's group and their collaborators. It states that if $K$ and $L$ are convex bodies, then for $q>0$,
	\begin{equation}\label{E:12}
		\tilde{V}_q((1-\lambda)K+\lambda L)^{\frac{1}{q}}\geq (1-\lambda)\tilde{V}_q(K)^{\frac{1}{q}}+\lambda\tilde{V}_q(L)^{\frac{1}{q}},\quad \forall~\lambda\in[0,1],
	\end{equation}
	with equality if and only if $K$ and $L$ are dilates of each other. One can also see these details in \cite{HYZ25}. Xi-Zhang \cite{XZ22} solved the conjecture when $q\in(0,1]$ from a geometric way. Recently, based on the analytic tools developed by Kolesnikov-Milman \cite{KM17, KM18}, Kolesnikov-Livshyts \cite{KL21}, and the estimate from Cordero-Erausquin and Rotem \cite{CR23}, Sadovsky-Zhang\cite{SZ25} confirmed the conjecture \eqref{E:12} for $0<q\leq n$ under symmetric assumptions, without the characterization of equality.
	
	We mainly focus on the $L_p$ Brunn-Minkowski inequality for dual quermassintegral, which is considered by Xi-Zhang \cite{XZ22} that, if it holds for any convex bodies $K$ and $L$ containing the origin in
	their interiors that
	\begin{equation}\label{E:15}
		\tilde{V}_q((1-\lambda)K+_p\lambda L)^{\frac{p}{q}}\geq (1-\lambda)\tilde{V}_q(K)^{\frac{p}{q}}+\lambda\tilde{V}_q(L)^{\frac{p}{q}},\quad\forall~\lambda\in[0,1]
	\end{equation}
	for $q\neq0$. They answered this question when $p\geq q$. The case $p<q$
	is likewise both interesting and challenging, which constitutes the main focus of our final result.
	
	For the case $p\geq1$, we employ a classic mathematical argument to tackle the problem. This result follows naturally from the Brunn-Minkowski inequality for $p=1$, i.e. the results established by Sadovsky-Zhang\cite{SZ25}. This also explain the reason that our constraints on the parameter $q$ and convex bodies coincide with those in their work.
	For $p\in(0,1)$, we use the near spherical uniqueness of Minkowski problem to prove the near spherical Brunn-Minkowski inequality for dual quermassintegrals. We first prove the equivalence between the $L_p$ Brunn-Minkowski inequality and the $L_p$ Minkowski inequality for dual quermassintegral (the explicit form of the $L_p$ Minkowski inequality is given in \eqref{E:2}). Building upon this equivalence, we show the $L_p$
	Minkowski inequality by using the near spherical uniqueness of the $L_p$
	Minkowski problem, and this argument requires one of the convex bodies to be sufficiently close to a ball. It is worth noting that \cite{PS26} adopted the inverse direction of perspective to this paper, utilizing the Brunn-Minkowski inequality to establish the uniqueness of solutions to the Minkowski problem. In a certain sense, this phenomenon boils down that the uniqueness of the Minkowski problem is essentially equivalent to the validity of the Brunn-Minkowski inequality. Our results are stated as follows.
	\begin{thm}
		\label{lp BMI no near ball}
		Let $n\geq2$. Then for $K,L\in\kne$, \eqref{E:15} holds when
		\begin{itemize}
			\item either, $p\geq1$ and $q\in(0,n]$,
			\item or, $p\in(0,1)$, $q\in(p,n]$, $K,L\in\mathcal{K}^{2,\alpha}_{+,e}$, and for some $\epsilon>0$, $\|h_K-1\|_{C^{2,\alpha}}\leq\epsilon$.
		\end{itemize}
	\end{thm}
	This paper is organized as follows. In Section \ref{Preliminaries}, we present some preliminaries for basic notations and convex bodies. For the $L_p$ dual Minkowski problem, in Section \ref{a priori estimate}, we establish $C^0$ estimates when $p\in(0,q), q\in (1,n]$, which is Theorem \ref{C0 estimate in introduction}, and in Section \ref{section uniqueness}, we derive the uniqueness results when the density function near the ball and prove Theorem \ref{uniqueness result}. In Section \ref{section lp dual Minkowski inequality}, we study the $L_p$ dual Brunn-Minkowski inequality \eqref{E:15} and give the proof of Theorem \ref{lp BMI no near ball}.
	
	\section{Preliminaries}\label{Preliminaries}
	In this section, some notations and introduction about convex bodies will be provided (see excellent reference \cite{G06,S93}) .
	
	\subsection{Basic Notations}
	Let $\mathbb{R}^n$ be the $n$-dimensional Euclidean space. The unit sphere in $\mathbb{R}^n$ is defined as $S^{n-1}$. $C(S^{n-1})$ represents the space of continuous functions on  $S^{n-1}$, and will always be viewed as equipped with the max-norm metric
	\begin{equation*}
		\|f-g\|_{\infty}=\max_{x\in S^{n-1}}|f(x)-g(x)|,
	\end{equation*}
	for $f,g\in C(S^{n-1})$. Denote by $C^+(S^{n-1})$ the set of positive continuous functions on $S^{n-1}$, and by $C^+_e(S^{n-1})$ the set of positive continuous even functions on $S^{n-1}$.
	
	A convex body $K$ in $\mathbb{R}^n$ is a compact convex set with nonempty interior. Let $\mathcal{K}_0^n$ be 
	a class of convex bodies that contain the origin in their interiors in $\mathbb{R}^n$, $\mathcal{K}_e^n$ be 
	a class of origin-symmetric convex bodies in $\mathbb{R}^n$,
	$\mathcal{K}_{+,e}^2$ be the subclass of $\kne$ with $C^2$ smooth boundary and strictly positive curvature in $\mathbb{R}^n$, and $\mathcal{K}_{+,e}^{2,\alpha}\subset\mathcal{K}_{+,e}^2$ with $C^{2,\alpha}$ smooth boundary.
	
	\subsection{Convex bodies}
	Let $K$ be a compact convex subset in $\mathbb{R}^n$, then the support function $h_K$ of $K$ is defined by
	\begin{equation}
		h_K(x)=\max\{x\cdot y: y\in K\},~~~~x\in\mathbb{R}^n.
	\end{equation}
	The support function $h=h_K: \mathbb{R}^n\rightarrow\mathbb{R}$ is a continuous function homogeneous of degree 1. Suppose $K$ contains the origin in its interior, then the radial function $\rho=\rho_K: \mathbb{R}^n\backslash\{0\}\rightarrow\mathbb{R}$ is defined by
	\begin{equation}
		\rho_K(x)=\max\{\lambda: \lambda x\in K\},~~~~x\in \mathbb{R}^n\backslash\{0\}.
	\end{equation}
	The radial function $\rho_K$ is a continuous function homogeneous of degree $-1$. 
	
	For each $f\in C^+(\sn)$, the Wulff shape $[f]$ generated by $f$ is the convex body defined by
	\begin{equation*}
		[f]=\{x\in\mathbb{R}^n: x\cdot u\leq f(u), \text{for all}~~ u\in S^{n-1}\}.
	\end{equation*}
	It is obvious that $h_{[f]}\leq f$ and $[h]=K$ for each $K\in\mathcal{K}_0^n$.
	
	Suppose $K_i\in\mathcal{K}_0^n$ is a sequence of convex bodies in $\mathbb{R}^n$, we say $K_i$ converges to a compact convex subset $K\subset\mathbb{R}^n$ with respect to the Hausdorff metric provided that
	\begin{equation*}
		\|h_{K_i}(u)-h_K(u)\|_{\infty}=\max_{u\in S^{n-1}}|h_{K_i}(u)-h_K(u)|\rightarrow 0,\quad i\rightarrow\infty.
	\end{equation*}
	If $K$ contains the origin in its interior, the above formula is equivalent to
	\begin{equation*}
		\|\rho_{K_i}(u)-\rho_K(u)\|_{\infty}=\max_{u\in S^{n-1}}|\rho_{K_i}(u)-\rho_K(u)|\rightarrow 0,\quad i\rightarrow\infty.
	\end{equation*}
	
	We use $\nu_K$ to denote the Gauss map that takes $x\in \partial K$ to its unique outer unit normal which is due to the convexity of $K$, and the map $\nu_K$ is almost everywhere defined on $\partial K$. We use $\nu_K^{-1}$ to denote the inverse Gauss map. Since $K$ is not assumed to be strictly convex, the map $\nu_K^{-1}$ is set-valued map and for each set $\eta\subset\sn$, 
	\begin{equation*}
		\nu_K^{-1}(\eta)=\{x\in\partial K:\nu_K(x)\in\eta\}.
	\end{equation*}
	We employ the following renormalization for the Gauss and inverse Gauss map to enhance clarity. For $u\in S^{n-1}$, if $\nu_K$ is well-defined at $\rho_K(u)u\in\partial K$, then we write $\alpha_K(u)$ for $\nu_K(\rho_K(u)u)$. The reverse radial Gauss mapping $\alpha^{-1}_K$ is defined as follow:
	\begin{equation*}
		\alpha^{-1}_{K}(\eta)=\{u\in\sn:\nu_K(\rho_K(u))\in\eta\}\quad\text{for any}~\eta\in\sn.
	\end{equation*}
	It is well known that the Jacobian of $\alpha^{-1}_{K}$ is
	\begin{equation*}
		|\text{Jac} \alpha_K^{-1}|=\frac{h\det(\nabla^2h+hI)}{(h^2+|\nabla h|^2)^{\frac{n}{2}}}.
	\end{equation*}
	\section{A priori estimates}\label{a priori estimate}
	In this section, we provided uniform $C^0$ estimates of solutions to the $L_p$ dual Minkowski problem with $p\in(0,q)$ and $q\in(1,n]$. The method is inspired by \cite{HL13}.
	
	\begin{lem}\label{lem}
		Assume that $h\in C^3(\sn)$ is a solution of \eqref{MA equation in introduction} for $q\in(1,n]$ and $p\in(0,q)$, then there is a positive constant $0<\alpha<1$ such that
		\begin{equation}\label{gdfg}
			\frac{|\nabla h|}{h^\alpha} \leq C,
		\end{equation}
		where $C$ depends only on $\|f\|_{C^1(\sn)},$  $\min\limits_{\sn}f$, $\max\limits_{\sn}f$, $\alpha, n, p$ and $q$.
	\end{lem}
	
	\begin{proof}
		Consider the testing function $G=  \frac{|\nabla h|}{h^\alpha}, $ at its maximum $u_0\in \sn$. Without loss of generality, we choose
		normal coordinates at $u_0$ such that
		\begin{align}
			h_1&=|\nabla h| \tag{\theequation a}\\
			h_i&=0 \quad  \text{for } i=2,\cdots, n-1. \tag{\theequation b}
		\end{align}
		Writing equation \eqref{MA equation in introduction} as
		$$
		\log\det (h_{ij}+h\delta_{ij})=\log\det(w_{ij}) =(p-1)\log h+\frac{n-q}2\log(h^2+|\nabla h|^2)+\log f,
		$$
		and differentiating it in direction $e_1\in T_{u_0}\sn,$ we obtain
		
		$$
		w^{ij}w_{ij1}=\frac{f_1}{f}+(p-1)\frac{h_1}{h}+(n-q)\frac{hh_1+h_1h_{11}}{h^2+h_1^2}.
		$$
		At $u_0$,
		$$
		0=(\log G)_i=\frac{h_lh_{li}}{|\nabla h|^2}-\alpha \frac{h_i}{h}=\frac{h_{1i}}{h_1}-\alpha \frac{h_i}{h},
		$$
		i.e.
		\begin{equation}\label{eq:hy}
			h_{11}=\alpha\frac{h_1^2}{h}, \quad  h_{1i}=0  \quad  \text{for } i\geq 2.
		\end{equation}
		We may as well assume that $[h_{ij}(u_0)]$ is diagonal.  Noticing that $(\log G)_{ij}$ is nonpositive definite at $u_0$,
		\begin{eqnarray}\label{nbv}
			0&\geq& w^{ij}(\log G)_{ij}
			=w^{ij}(\frac{h_{li}h_{lj}}{h_1^2}+\frac{h_{1ij}}{h_1}-\frac{2h_{1i}h_{1j}}{h^2_1}-\frac{\alpha h_{ij}}{h}+\frac{\alpha h_ih_j}{h^2})\nonumber\\
			&\geq& w^{ij}(\frac{h_{1ij}}{h_1}-\frac{h_{1i}h_{1j}}{h^2_1}-\frac{\alpha h_{ij}}{h}+\frac{\alpha h_ih_j}{h^2})\nonumber\\
			&=&\frac 1{h_1}w^{ij}w_{ij1}-w^{11}+(\alpha-\alpha^2)w^{11}\frac{h_1^2}{h^2} -\frac{\alpha (n-1)}{h}+\alpha\sum\limits_i w^{ii},\nonumber
		\end{eqnarray}
		where we have used
		$h_{1ij}=w_{1ij}-\delta_{1i}h_j=w_{ij1}-\delta_{1i}h_j$  and $w_{ij}=h_{ij}+h\delta_{ij}.$
		Since
		\begin{eqnarray*}
			w^{11}&=&\frac{1}{h_{11}+h}=\frac{h}{\alpha h_1^2+h^2}\\
			w^{ij}w_{ij1}&=&\frac {(p-1)h_1}{h}+\frac{f_1}{f}+(n-q)\frac{hh_1+h_1h_{11}}{h^2+h_1^2}.
		\end{eqnarray*}
		Thus for $q\leq n,$
		
		\begin{equation} \label{eq:l3}
			\begin{split}
				0\geq& \frac{p-1}{h}+\frac{n-q}h\frac{h^2+\alpha h_1^2}{h^2+h_1^2}+\frac{f_1}{fh_1}+\frac{(\alpha-\alpha^2)h_1^2-h^2}{\alpha
					hh_1^2+h^3}-\frac{\alpha (n-1)}{h}+\alpha\sum\limits_i w^{ii}\\
				\geq& \frac{p-1}{h}+\frac{f_1}{fh_1}+\frac{(\alpha-\alpha^2)h_1^2-h^2}{\alpha
					hh_1^2+h^3}-\frac{\alpha (n-1)}{h}+\alpha\sum\limits_i w^{ii}.
			\end{split}
		\end{equation}
		
		We divide it into two cases to deal with the inequality according to
		value of $h(u_0).$ It is to say that there exists a positive
		constant $M$ depending only on $p, q, n, f$ and to be determined such
		that
		\begin{itemize}
			\item Case 1: $h(u_0)\leq M;$
			\item Case 2: $h(u_0)\geq M.$
		\end{itemize}
		Case 1: $h(u_0)\leq M:$
		Subcase$(1_A)$: $h_1\leq \frac{h}{\alpha}, $ then
		$\frac{h_1}{h^\alpha}\leq \frac{h^{1-\alpha}}{\alpha} \leq
		\frac{M^{1-\alpha}}{\alpha}.$
		
		Subcase$(1_B)$: $h_1\geq \frac{h}{\alpha}, $ then $\frac{h^2}{\alpha
			h_1^2+h^2}=\frac{1}{\alpha
			(\frac{h_1}{h})^2+1}\leq\frac{\alpha}{1+\alpha}<\alpha,$ combining
		this and \eqref{eq:l3}, we have
		
		\begin{equation} \label{eq:l4}
			\begin{split}
				0&\geq \frac{p-1}{h}+\frac{f_1}{fh_1}+\frac{(\alpha-\alpha^2)h_1^2-h^2}{\alpha
					hh_1^2+h^3}-\frac{\alpha (n-1)}{h}+\alpha\sum\limits_i w^{ii}\\
				&=\frac{p-1}{h}+\frac{f_1}{fh_1}+\frac{1}{h}\left[1-\alpha-\frac{(2-\alpha)h^2}{\alpha
					h_1^2+h^2}\right]-\frac{\alpha (n-1)}{h}+\alpha\sum\limits_i w^{ii}\\
				&\geq\frac{1}{h}[p-(2+n)\alpha+\alpha^2]+\frac{f_1}{fh_1}.
			\end{split}
		\end{equation}
		Here we set $p\geq(n+3)\alpha,$ then it yields
		$$
		0\geq\frac{\alpha^2}{h}-\frac{|\nabla f|}{fh_1},
		$$
		so
		\begin{equation}
			\frac{h_1}{h^\alpha}\leq \frac{|\nabla f|}{\alpha^2f}M^{1-\alpha}.
		\end{equation}
		Case 2: $h(u_0)\geq M.$ We divide Case 2 in three subcases as follows:
		
		Subcase$(2_A)$: $h_1\geq C_1h$, where $C_1$ is a positive constant, and chosen so that
		\begin{equation}
			\label{choose C1}
			\frac{\alpha(n-1)}{2}C_1^{\frac{q-1}{n-1}}\left(1+\frac{1}{C_1^2}\right)^{-\frac{n-q}{2(n-1)}}M^{\frac{q-p}{n-1}}(\max f)^{-\frac{1}{n-1}}>\max\left\{\frac{2-\alpha}{C_1\alpha},\frac{|\nabla f|}{f}\right\}.
		\end{equation}
		Since $h_1\geq C_1h$, thus
		\begin{equation}\label{special estimate in case 3}
			\frac{h}{\alpha h_1^2+h^2}\leq\frac{h_1}{C_1(\alpha h_1^2+h^2)}\leq\frac{1}{C_1h_1}\frac{h_1^2+\frac{1}{\alpha}h^2}{\alpha h_1^2+h^2}=\frac{1}{\alpha C_1h_1}.
		\end{equation}
		Combining \eqref{special estimate in case 3} and \eqref{eq:l3}, we deduce that
		\begin{eqnarray*}
			0&\geq&\frac{p-1}{h}+\frac{f_1}{fh_1}+\frac{1-\alpha}{h}-(2-\alpha)\frac{h}{\alpha
				h_1^2+h^2}-\frac{\alpha (n-1)}{h}+\alpha\sum\limits_i w^{ii}\\
			&\geq&\frac{p-n\alpha}{h}+\left(\frac{f_1}{fh_1}+\frac
			{\alpha}{2}\sum\limits_i w^{ii}\right)+ \left(\frac{\alpha}{2}\sum\limits_i w^{ii}-\frac{2-\alpha}{\alpha C_1h_1}\right),
		\end{eqnarray*}
		which will lead to a contradiction if 
		\begin{equation}\label{contradiction in 2A}
			\left(\frac{f_1}{fh_1}+\frac
			{\alpha}{2}\sum\limits_i w^{ii}\right)+ \left(\frac{\alpha}{2}\sum\limits_i w^{ii}-\frac{2-\alpha}{\alpha C_1h_1}\right)>0.
		\end{equation}
		Indeed, \eqref{eq:hy} tells us that $w_{ij}=h_{ij}+h\delta_{ij}$
		may be diagonal at $u_0$,
		\begin{equation}\label{w}
			\sum\limits_i w^{ii}\geq (n-1)\det(w_{ij})^{-\frac
				1{n-1}}=(n-1)(h^{p-1}(h^2+h_1^2)^{\frac {n-q}2}f)^{-\frac 1{n-1}}.   
		\end{equation}
		Then \eqref{choose C1}, \eqref{w} and the fact $\frac{h_1}{C_1}\geq h\geq M$
		give that
		\begin{align*}
			\frac{f_1}{fh_1}+\frac {\alpha}{2}\sum\limits_i w^{ii}&\geq
			-\frac{|\nabla f|}{fh_1}+\frac {\alpha(n-1)}{2}(h^{p-1}(h^2+h_1^2)^{\frac {n-q}2}f)^{-\frac 1{n-1}}\\
			&\geq\frac{1}{h_1}\left(\frac{-|\nabla f|}{f}+\frac{\alpha(n-1)}{2}h_1h^{-\frac{p-1}{n-1}}\left(1+\frac{1}{C_1^2}\right)^{-\frac{n-q}{2(n-1)}}h_1^{-\frac{n-q}{n-1}}f^{-\frac{1}{n-1}}\right)\\
			&\geq \frac{1}{h_1}\left(-\frac{|\nabla f|}{f}+\frac{\alpha(n-1)}{2}C_1^{\frac{q-1}{n-1}}\left(1+\frac{1}{C_1^2}\right)^{-\frac{n-q}{2(n-1)}}M^{\frac{q-p}{n-1}}f^{-\frac{1}{n-1}}\right)>0
		\end{align*}
		provided $1<q\leq n$ and $p<q$.
		On the other hand, by \eqref{choose C1} we have
		\begin{align*}
			\frac {\alpha}{2}\sum\limits_i w^{ii}-\frac{2-\alpha}{\alpha C_1h_1}&\geq \frac{1}{h_1}\left(\frac
			{\alpha(n-1)}{2}h_1\left(h^{p-1}(h^2+h_1^2)^{\frac {n-q}2}f\right)^{-\frac 1{n-1}}-\frac{2-\alpha}{C_1\alpha}\right)\\
			&\geq \frac{1}{h_1}\left(\frac{\alpha(n-1)}{2}h_1h^{-\frac{p-1}{n-1}}\left(1+\frac{1}{C_1^2}\right)^{-\frac{n-q}{2(n-1)}}h_1^{-\frac{n-q}{n-1}}f^{-\frac{1}{n-1}}-\frac{2-\alpha}{C_1\alpha}\right)\\
			&\geq\frac{1}{h_1}\left(\frac{\alpha(n-1)}{2}C_1^{\frac{q-1}{n-1}}\left(1+\frac{1}{C_1^2}\right)^{-\frac{n-q}{2(n-1)}}M^{\frac{q-p}{n-1}}f^{-\frac{1}{n-1}}-\frac{2-\alpha}{C_1\alpha}\right)\\
			&>0.
		\end{align*}
		The above two inequalities imply \eqref{contradiction in 2A}.
		
		Subcase$(2_B)$: $C_2h^{1-\frac{q-p}{n-1}}\leq h_1\leq C_1h$ where $C_2$ is chosen
		such that
		\begin{equation}\label{sd}
			\frac{\alpha(n-1)}{2}\left(1+C_1^2\right)^{-\frac{n-q}{2(n-1)}}C_2>\max\frac{|\nabla f|}{f^{1-\frac{1}{n-1}}}.
		\end{equation}
		Furthermore, we now choose $M\gg C_1\gg1$ such that
		\begin{equation}\label{bnm}
			\frac{\alpha(n-1)}{2}\left(1+C_1^2\right)^{-\frac{n-q}{2(n-1)}}M^{\frac{q-p}{n-1}}(\max f)^{-\frac{1}{n-1}}>2.
		\end{equation}
		From \eqref{eq:l3}, we get
		\begin{eqnarray*}
			0&\geq&\frac{p-1}{h}+\frac{f_1}{fh_1}+\frac{1-\alpha}{h}-(2-\alpha)\frac{h}{\alpha
				h_1^2+h^2}-\frac{\alpha (n-1)}{h}+\alpha\sum\limits_i w^{ii}\\
			&\geq&\frac{p-n\alpha}{h}+\left(\frac{f_1}{fh_1}+\frac
			{\alpha}{2}\sum\limits_i w^{ii}\right)+ \left(\frac
			{\alpha}{2}\sum\limits_i w^{ii}-\frac{2}{h}\right),
		\end{eqnarray*}
		which leads to a contradiction provided that
		\begin{eqnarray}\label{eq:xulu}
			\left(\frac{f_1}{fh_1}+\frac {\alpha}{2}\sum\limits_i
			w^{ii}\right)+ \left(\frac {\alpha}{2}\sum\limits_i
			w^{ii}-\frac{2}{h}\right)>0.
		\end{eqnarray}
		Then combining \eqref{w}, \eqref{sd} and $C_2h^{1-\frac{q-p}{n-1}}\leq h_1\leq C_1h$
		\begin{align*}
			\frac{f_1}{fh_1}+\frac {\alpha}{2}\sum\limits_i w^{ii}&\geq
			-\frac{|\nabla f|}{fh_1}+\frac {\alpha(n-1)}{2}(h^{p-1}(h^2+h_1^2)^{\frac {n-q}2}f)^{-\frac 1{n-1}}\\
			&\geq\frac{1}{h_1}\left(-\frac{|\nabla f|}{f}+\frac{\alpha(n-1)}{2}\left(1+C_1^2\right)^{-\frac{n-q}{2(n-1)}}h_1h^{-\frac{n-q+p-1}{n-1}}f^{-\frac{1}{n-1}}\right)\\
			&\geq\frac{1}{h_1}\left(-\frac{|\nabla f|}{f}+\frac{\alpha(n-1)}{2}\left(1+C_1^2\right)^{-\frac{n-q}{2(n-1)}}C_2f^{-\frac{1}{n-1}}\right)>0
		\end{align*}
		provided $p<q$ and $q\leq n$.
		On the other hand, by \eqref{bnm} and $h_1\leq C_1h$, we have
		\begin{align*}
			\frac {\alpha}{2}\sum\limits_i w^{ii}-\frac{2}{h}&\geq \frac
			{\alpha(n-1)}{2}\left(h^{p-1}(h^2+h_1^2)^{\frac {n-q}2}f\right)^{-\frac 1{n-1}}-\frac{2}{h}\\
			&\geq\frac{1}{h}\left(\frac{\alpha(n-1)}{2}\left(1+C_1^2\right)^{-\frac{n-q}{2(n-1)}}h^{1-\frac{n-q+p-1}{n-1}}f^{-\frac{1}{n-1}}-2\right)\\
			&\geq\frac{1}{h}\left(\frac{\alpha(n-1)}{2}\left(1+C_1^2\right)^{-\frac{n-q}{2(n-1)}}M^{\frac{q-p}{n-1}}f^{-\frac{1}{n-1}}-2\right)>0
		\end{align*}
		provided $p<q$ and $q\leq n$.  The above inequalities directly deduce \eqref{eq:xulu}.
		
		Subcase$(2_C)$: $h_1\leq C_2h^{1-\frac{q-p}{n-1}}, $ then
		\begin{equation}
			\frac{h_1}{h^\alpha}\leq C_2\frac{h^{1-\alpha}}{h^{\frac{q-p}{n-1}}} \leq
			\frac{C_2}{M^{\frac{q-p}{n-1}}}h^{1-\alpha}.
		\end{equation}
		So we get
		\begin{equation}\label{mnbj}
			\frac{|\nabla h|}{h^\alpha}\leq \frac{C_2}{M^{\frac{q-p}{n-1}}}h^{1-\alpha}, \quad\text { for any }x\in \sn.
		\end{equation}
		Let $\Gamma$ be a great circle connecting $x_1$ and $x_2$ with
		length less than or equal to $\pi,$ here $h(x_1)=\min h$ and
		$h(x_2)=\max h.$ Then
		\begin{equation}\label{ux}
			h(x_2)^{1-\alpha}-h(x_1)^{1-\alpha}=\int\limits_\Gamma
			dh^{1-\alpha}\leq (1-\alpha)\int\limits_\Gamma\frac{|\nabla
				h|}{h^\alpha}\leq  \frac{C_2(1-\alpha)\pi}{M^{\frac{q-p}{n-1}}}h(x_2)^{1-\alpha}.
		\end{equation}
		From equation \eqref{MA equation in introduction}, we have for $p<q$
		\begin{eqnarray*}
			h(x_1)=\min h\leq (\max f)^{\frac{1}{q-p}}.
		\end{eqnarray*}
		This with \eqref{ux} imply the upper bound of $h$ provided that
		\begin{equation}
			\label{choose C2}
			\frac{C_2(1-\alpha)\pi}{M^{\frac{q-p}{n-1}}}\leq \frac 1 2.
		\end{equation}
		So we get our desired upper bound for $\frac{h_1}{h^\alpha}$ from
		\eqref{mnbj}.
		
		Finally, we would like to claim that our requirements for the constants $C_1$, $C_2$ and $M$ are not contradictory in Case 2, i.e., \eqref{choose C1} to \eqref{bnm}, and \eqref{sd} to \eqref{choose C2}. 
		
		For \eqref{choose C1} to \eqref{bnm}, we first ask that, for any $C_1>0$, choose $M$ large so that $\eqref{bnm}$ holds. Then,
		\begin{align*}
			\text{left-hand side of \eqref{choose C1}}&=\frac{\alpha(n-1)}{2}C_1^{\frac{q-1}{n-1}}\left(1+\frac{1}{C_1^2}\right)^{-\frac{n-q}{2(n-1)}}M^{\frac{q-p}{n-1}}(\max f)^{-\frac{1}{n-1}}\\
			&=\frac{\alpha(n-1)}{2}C_1\left(1+C_1^2\right)^{-\frac{n-q}{2(n-1)}}M^{\frac{q-p}{n-1}}(\max f)^{-\frac{1}{n-1}}\\
			&=C_1\cdot\text{left-hand side of \eqref{bnm}}.
		\end{align*}
		Thus, we only need to choose $C_1$ large enough, so that $\eqref{choose C1}$ is established.
		
		For \eqref{sd}, \eqref{choose C2}, we ask that, for any $C_1>0$, choose $C_2$ large enough, so that $\eqref{sd}$ holds. By choosing $M$ large again, \eqref{choose C2} is true.
		We finish the proof.
	\end{proof}
	
	Lemma~\ref{lem} implies the upper bound of $h$.
	\begin{lem}\label{lem5.2}(Upper bound of Theorem \ref{C0 estimate in introduction})\label{lem1}
		Assume that $h$ is a solution to equation \eqref{MA equation in introduction} for $p\in(0,1]$ and $q\in(1,n],$ then there exists constant $C(f,n,p,q)$ such that
		\begin{equation}\label{SB}
			\sup\limits_{\sn} h\leq C.
		\end{equation}
	\end{lem}
	\begin{proof}
		Let $\Gamma$ be a great circle connecting $x_1$ and $x_2$ with
		length less than or equal to $\pi,$ where $h(x_1)=\min h$ and
		$h(x_2)=\max h.$ Then by \eqref{gdfg} in Lemma~\ref{lem},
		\begin{equation*}
			h(x_2)^{1-\alpha}-h(x_1)^{1-\alpha}=\int\limits_\Gamma
			dh^{1-\alpha}\leq (1-\alpha)\int\limits_\Gamma\frac{|\nabla
				h|}{h^\alpha}\leq C.
		\end{equation*}
		On the other hand, from the equation \eqref{MA equation in introduction}, we have $h_{ij}\geq 0$ at $x_1$ and
		\begin{eqnarray*}
			h(x_1)=\min h\leq (\max f)^{\frac{1}{q-p}}.
		\end{eqnarray*}
		Thus
		$$
		h(x_2)\leq C.
		$$
	\end{proof}
	We next use the following trick to get the lower bound of volume.
	\begin{lem}(Lower bound of volume in Theorem \ref{C0 estimate in introduction})
		Suppose $h$ is a solution to the equation \eqref{MA equation in introduction} when $p\in(0,1]$ and $q\in(1,n]$, then there exists constant $C(f,n,p,q)$ such that
		\[
		V_n([h])>C.
		\]
	\end{lem}
	\begin{proof}
		Let $K=[h]$. Then by H\"older's inequality,
		\begin{align}\label{using holder inequality}
			V_n(K)=\frac{1}{n}\int_{\sn}\rho_K^n\geq\frac{1}{n}(\omega_{n-1})^{\frac{q-n}{q}}\left(\int_{\sn}\rho_K^q\right)^{\frac{n}{q}}
		\end{align}
		provided $0<q\leq n$. Since $h$ is a solution to equation \eqref{MA equation in introduction}, then
		\begin{equation}\label{from radial to support}
			\begin{split}\int_{\alpha^{-1}_K(\eta)}\rho^q(\xi)d\xi&=\int_{\eta}\rho^q(\alpha_K^{-1}(x))\text{Jac}(\alpha_K^{-1})dx\\
				&=\int_{\eta}\rho^{q-n}(\alpha_K^{-1}(x))h(x)\det(h_{ij}+h\delta_{ij})dx\\
				&=\int_{\eta}h^{p}(x)f(x)dx
			\end{split}
		\end{equation}
		for any $\eta\subset\sn$. By the definition of support function,  
		\begin{equation*}
			h(x)\geq h_{\max}\langle x,u_0\rangle_+
		\end{equation*}
		with $h(u_0)=h_{\max}$. Moreover,  we have $h_{ij}\leq0$ at $u_0$. From the equation \eqref{MA equation in introduction}, we have
		\begin{equation*} 
			h(u_0)^{n-1}\geq f(u_0)h(u_0)^{p-1+n-q},  
		\end{equation*}
		which means 
		\begin{equation}
			\label{hmax has lower bound}
			h_{\max}\geq C
		\end{equation} 
		when $p<q$.
		From \eqref{from radial to support}, we have
		\begin{equation}\label{integral of rhoq is bounded}
			\begin{split}\int_{\sn}\rho^q(\xi)d\xi&=\int_{\sn}h^p(x)f(x)dx\\
				&\geq(\min_{\sn} f)\int_{\sn}h^p(x)dx\\
				&\geq (\min_{\sn} f)h_{\max}^p\int_{\sn}\langle x,u_0\rangle_+^pdx\\
				&\geq C,
			\end{split}
		\end{equation}
		where the last inequality is because \eqref{hmax has lower bound}, and the integral $\int_{\sn}\langle x,u_0\rangle_+^pdx$ is positive constant. Combining \eqref{using holder inequality} and \eqref{integral of rhoq is bounded}, the volume of $K$ is bounded from below.
	\end{proof}
	
	\section{Uniqueness near the ball}\label{section uniqueness}
	After obtaining $C^0$ estimates of the form given in Section \ref{a priori estimate}, we apply the method in B\"or\"oczky-Chen-Liu-Saroglou \cite{BCLS-1} to derive the uniqueness result Theorem \ref{uniqueness result}. For the sake of completeness, we repeat some lemmas and steps from \cite{BCLS-1}. 
	
	\begin{lem}(BCLS\cite{BCLS-1})
		\label{lemma 7.1 in BCLS}
		Let $\lambda>1$, $p\in(-1,1)$, $q_m>1$ tend to $n$, and $K_m\in\kno$ tend to $K_\infty\in\kno$ as $m$ tends to infinity where $\lambda^{-1}\mathcal{H}^{n-1}\leq\tilde{C}_{p,q_m,K_m}\leq\lambda\mathcal{H}^{n-1}$. Then $\tilde{C}_{p,q_m,K_m}$ tends weakly to $S_{p,K_\infty}$.
	\end{lem}
	Now we prove Theorem \ref{uniqueness result}. 
	\begin{proof}[Proof of Theorem \ref{uniqueness result}]
		We begin first with the case $p\in[1,q)$ and $q\in(1,n]$. Consider a family of operators $\Psi:C^{2,\alpha}_+(\sn)\rightarrow C^{\alpha}(\sn)$ defined by
		\begin{equation}
			\Psi(h):=h^{1-p}\left(h^2+|\nabla h|^2\right)^{\frac{q-n}{2}}\det(\nabla^2h+hI).
		\end{equation}
		It is easy to calculate that the linear operator of $\Psi$ at $h_0\equiv1$, denoted by $L_{h_0}:C^{2,\alpha}(\sn)\rightarrow C^\alpha(\sn)$ is
		\begin{equation*}
			L_{h_0}(\psi)=\Delta_{\sn}\psi+(q-p)\psi.
		\end{equation*}
		The spherical Laplacian $\Delta_{\sn}$ admits discrete eigenvalues 
		\begin{equation*}
			-\lambda_k=k^2+(n-2)k,\quad k=0,1,2,\cdots.
		\end{equation*}
		It is easy to find that $L_{h_0}$ is invertible, provided $p\in[1,q)$ and $q\in(1,n]$. Indeed, the eigenvalues $\tilde{\lambda}_k$ of $L_{h_0}$ are
		\begin{equation*}
			\tilde{\lambda}_k=-k^2-(n-2)k+q-p,\quad k=0,1,2,\cdots,
		\end{equation*}
		and $\tilde{\lambda}_0=q-p>0$, $\tilde{\lambda}_1=1-n+q-p<0$ and $\tilde{\lambda}_k<0$ when $k=2,3,\cdots$. By inverse function theorem, there exists a neighborhood $\mathcal{N}$ of $h_0$ such that, if $h_1,h_2\in\mathcal{N}$ and 
		\begin{equation*}
			\Psi(h_1)=\Psi(h_2),
		\end{equation*}
		then $h_1=h_2$. Due to this, to prove Theorem \ref{uniqueness result} it is sufficiently to prove that, there exists constant $\epsilon>0$ depending only on $n,\alpha, p$ and $q$ such that if the smooth function $f$ satisfies $\|f-1\|_{C^{\alpha}(\sn)}<\epsilon$, then for any convex body $K$ satisfying 
		\begin{equation*}
			d\tilde{C}_{p,q}(K,\cdot)=fd\mathcal{H}^{n-1},
		\end{equation*}
		then the support function $h_K\in\mathcal{N}$. Suppose not, then there exist a sequence function $f_i\rightarrow1$ in $C^{\alpha}(\sn)$ and a sequence of convex bodies $K_i$ such that
		\begin{equation*}
			d\tilde{C}_{p,q}(K_i,\cdot)=f_id\mathcal{H}^{n-1},
		\end{equation*}
		but $h_{K_m}\notin\mathcal{N}$ when $m$  is sufficiently large. By Theorem \ref{C0 estimate in introduction}, we have $\|h_{Km}\|_{L^{\infty}(\sn)}\leq C$ and $|K_m|\geq C^{-1}$ for some positive constant $C$. Thus, using Blaschke selection theorem, $K_m$ converges to a convex body $K_\infty$ up to a subsequence. By the weak convergence of $\tilde{C}_{p,q}(K,\cdot)$, then
		\begin{equation*}
			d\tilde{C}_{p,q}(K_{\infty})=\mathcal{H}^{n-1}.
		\end{equation*}
		By using the uniqueness result in \cite{CL21} when $p>1$ and $q\leq n$, $K_\infty$ must be the ball centered at origin.  That is to say,
		\begin{equation*}
			\|h_K-1\|_{L^{\infty}(\sn)}\rightarrow0\quad \text{as}\ m\rightarrow\infty.
		\end{equation*}
		By Caffarelli's $C^{2,\alpha}$ results \cite{Ca90,C90}, $\|h_{K_{m}}\|_{C^{2,\alpha}(\sn)}\leq C_0$ for some positive constant $C_0$. Thus,
		\begin{equation*}
			\|h_{K_m}^{p-1}(h_{K_m}^2+|\nabla h_{K_m}|^2)^{\frac{n-q}{2}}f_m-1\|_{C^{\alpha}(\sn)}<\epsilon_0\rightarrow0\ \text{as}\ \epsilon\rightarrow0.
		\end{equation*}
		Let $\omega=h_{K_m}-1$, $a_{ij}=\int_0^1U_t^{ij}dt$, where $U_{t}^{ij}$ is the cofactor matrix of 
		\begin{equation*}
			\nabla^2(1+t\omega)+(1+t\omega)I.
		\end{equation*}
		Then
		\begin{align*}
			h_{K_m}^{p-1}(h_{K_m}^2+|\nabla h_{K_m}|^2)^{\frac{n-q}{2}}f_m-1&=\det(\nabla^2h_{K_m}+h_{K_m}I)-1\\
			&=\sum_{i,j=1}^na_{ij}(\nabla^2_{ij}\omega+\omega\delta_{ij}).
		\end{align*}
		Note that $C^{-1}I\leq a_{ij}\leq CI$ for some $C>0$, thus $\omega$ satisfies a uniformly elliptic linear equation. By Schauder estimates, 
		\begin{align*}
			\|h_{K_m}-1\|_{C^{2,\alpha}}\leq C\left(\|h_{K_m}-1\|_{L^\infty}+\|h_{K_m}^{p-1}(h_{K_m}^2+|\nabla h_{K_m}|^2)^{\frac{n-q}{2}}f_m-1\|_{C^{\alpha}}\right)\rightarrow0
		\end{align*}
		as $m\rightarrow\infty$. This is a contradiction with $h_{K_m}\notin\mathcal{N}$. 
		
		As for the case when $n-\epsilon<q\leq n$ and $p\in(0,1)$, we need to use the Lemma \ref{lemma 7.1 in BCLS} and the uniqueness result in Brendle-Choi-Daskalopoulos \cite{BCD17}. One can also see these details in \cite{BCLS-1}.
		We complete the proof. 
	\end{proof}
	
	\section{The $L_p$ dual Brunn-Minkowski inequality}\label{section lp dual Minkowski inequality}
	\subsection{Case $p\geq1$}
	Recently, Sadovsky-Zhang \cite{SZ25} proved that Lutwak's conjecture holds in the symmetric case. In this subsection, we use the classical trick to prove the $L_p$ Brunn-Minkowski inequality when $p\geq 1$ which is stated in Theorem \ref{lp BMI no near ball}.
	
	We first states Sadovsky-Zhang's result in \cite{SZ25}.
	\begin{thm}[Sadovsky-Zhang \cite{SZ25}]
		Let $K,L$ be origin-symmetric convex bodies in $\rn$. Then for $0<q\leq n$,
		\begin{equation}
			\label{BMI}
			\tilde{V}_q((1-\lambda)K+\lambda L)^{\frac{1}{q}}\geq(1-\lambda)\tilde{V}_q(K)^{\frac{1}{q}}+\lambda\tilde{V}_q(L)^{\frac{1}{q}} \quad \forall~\lambda\in [0,1].
		\end{equation}
		with equality if $K$ and $L$ are dilates of each other.   
	\end{thm}
	We are now in a position to prove the first item in Theorem \ref{lp BMI no near ball} by a classical trick.

	\begin{proof}[Proof of the first item in Theorem \ref{lp BMI no near ball}]
		Let $\bar{K}=\tilde{V}_q(K)^{-\frac{1}{q}}K$, $\bar{L}=\tilde{V}_q(L)^{-\frac{1}{q}}L$, and $\lambda=\frac{\tilde{V}_q(L)^{\frac{p}{q}}}{\tilde{V}_q(K)^{\frac{p}{q}}+\tilde{V}_q(L)^{\frac{p}{q}}}$. 
		By Jensen's inequality, for $p\geq1, \lambda\in[0,1]$,
		\begin{equation*}
			h_{(1-\lambda)K+_p\lambda L}=[(1-\lambda)h_K^p+\lambda h_L^p]^{\frac{1}{p}}\geq (1-\lambda)h_K+\lambda h_L=h_{(1-\lambda)K+\lambda L},
		\end{equation*}
		thus 
		\begin{equation}\label{sub}
			(1-\lambda)K+\lambda L\subseteq (1-\lambda)K+_p\lambda L.
		\end{equation}
		Combining with \eqref{BMI} and \eqref{sub},
		\begin{align*}
			\frac{\tilde{V}_q(K+_pL)^{\frac{1}{q}}}{(\tilde{V}_q(K)^{\frac{p}{q}}+\tilde{V}_q(L)^{\frac{p}{q}})^{\frac{1}{p}}}&=\tilde{V}_q\left(\frac{K+_pL}{(\tilde{V}_q(K)^{\frac{p}{q}}+\tilde{V}_q(L)^{\frac{p}{q}})^{\frac{1}{p}}}\right)^{\frac{1}{q}}\\
			&=\tilde{V}_q\left((1-\lambda)\bar{K}+_p\lambda\bar{L}\right)^{\frac{1}{q}}\\
			&\geq \tilde{V}_q\left((1-\lambda)\bar{K}+\lambda\bar{L}\right)^{\frac{1}{q}}\\
			&\geq (1-\lambda)\tilde{V}_q(\bar{K})^{\frac{1}{q}}+\lambda\tilde{V}_q(\bar{L})^{\frac{1}{q}}=1.
		\end{align*}
		Hence for $p\geq 1, q\in(0,n]$,
		$$\tilde{V}_q(K+_p L)^{\frac{p}{q}}\geq \tilde{V}_q(K)^{\frac{p}{q}}+\tilde{V}_q(L)^{\frac{p}{q}}.$$
		By the homogeneity of the above inequality, we immediately get
		$$\tilde{V}_q((1-\lambda)K+_p\lambda L)^{\frac{p}{q}}\geq(1-\lambda)\tilde{V}_q(K)^{\frac{p}{q}}+\lambda\tilde{V}_q(L)^{\frac{p}{q}}.$$
	\end{proof}
	\subsection{Case $p\in(0,1)$}
	In this subsection, our goal is to prove the near spherical $L_p$ dual Brunn-Minkowski inequality when $p\in(0,1)$ which is stated in Theorem \ref{lp BMI no near ball}. Inspired by the method in \cite{BLYZ12,CHLL20}, we first obtain the $L_p$ dual Minkowski inequality, which is Theorem \eqref{thm 10}. Then in Theorem \ref{lem1} we prove the equivalence between the $L_p$ Brunn-Minkowski inequality \eqref{E:15} and the $L_p$ Minkowski inequality \eqref{E:2} for dual quermassintegrals. Finally, we prove Theorem \ref{lp BMI no near ball}.
	
	We first state our result on the $L_p$ Minkowski inequality for dual quermassintegrals.
	\begin{thm}\label{thm 10}($L_p$ dual Minkowski inequality)
		For $n\geq2$, $\alpha\in(0,1)$, $p\in(0,1)$ and $q\in(p,n]$. There exists a constant $\epsilon_0\in(0,1)$ that depends only on $n, \alpha, p$ and $q$ such that, if $K,L\in\mathcal{K}_{+,e}^{2,\alpha}$ with $\|h_K-1\|_{C^{2,\alpha}}\leq \epsilon_0$, then 
		\begin{equation}\label{E:2}
			\left(\int_{\sn}\left(\frac{h_L}{h_K}\right)^pd\bar{C}_{q,K}\right)^{\frac{1}{p}}\geq\left(\frac{\tilde{V}_q(L)}{\tilde{V}_q(K)}\right)^{\frac{1}{q}},
		\end{equation}
		where $\bar{C}_{q,K}:=\frac{\tilde{C}_{q,K}}{\tilde{V}_q(K)}$.
	\end{thm}
	
	To prove the Theorem \ref{thm 10}, we first consider a minimization problem.
	\begin{lem}\label{lem7.5}
		The minimization problem
		\begin{equation*}
			\min_{\phi\in C_{+,e}^{2,\alpha}(\sn)}\Phi(\phi)
		\end{equation*}
		has a minimizer $h_{L_0}$, which is the support function of some $L_0\in\mathcal{K}_{+,e}^{2,\alpha}$. The functional $\Phi: C_{+,e}^{2,\alpha}(\sn)\rightarrow\mathbb{R}$ is defined by
		\begin{equation*}
			\Phi(\phi)=\frac{1}{\tilde{V}_q([\phi])^{\frac{p}{q}}}\int_{\mathbb{S}^{n-1}}\left(\frac{\phi}{h_K}\right)^pd\tilde{C}_{q,K}
		\end{equation*}
		for the given $K\in\mathcal{K}_{+,e}^{2,\alpha}$.
	\end{lem}
	\begin{proof}
		Let $h_{[\phi]}$ be the support function of convex body $[\phi]$. Since $\phi\in C_{+,e}^{2,\alpha}(\sn)$, $[\phi]\in\mathcal{K}_{+,e}^{2,\alpha}$ and $h_{[\phi]}\in C_{+,e}^{2,\alpha}(\sn)$. Moreover,
		\begin{equation*}
			h_{[\phi]}\leq\phi, \tilde{V}_q([\phi])=\tilde{V}_q(h_{[\phi]})\Rightarrow\Phi(h_{[\phi]})\leq\Phi(\phi)
		\end{equation*}
		which means that the minimizer of $\Phi(h_{[\phi]})$ is actually the minimiser of $\Phi(\phi)$. Note that $\tilde{V}_q$ is homogeneous of degree $q$ which yields
		\begin{equation*}
			\Phi(t\phi)=\frac{1}{\tilde{V}_q([t\phi])^{\frac{p}{q}}}\int_{\mathbb{S}^{n-1}}\left(\frac{t\phi}{h_K}\right)^pd\tilde{C}_{q,K}=\Phi(\phi).
		\end{equation*}
		This fact allows us to add a limitation of $\tilde{V}_q(L)=1$. Therefore,
		\begin{align*}
			s:=&\min_{\phi\in C^+_e(\mathbb{S}^{n-1})}\Phi(\phi)\\
			=&\min_{h_L\in C^+_e(\mathbb{S}^{n-1})}\Phi(h_L)\\
			=&\min\left\{ \int_{\mathbb{S}^{n-1}}\left(\frac{h_L}{h_K}\right)^pd\tilde{C}_{q,K}: L\in \mathcal{K}_e^n, \tilde{V}_q(L)=1
			\right\}.
		\end{align*}
		
		Denote $\int_{\mathbb{S}^{n-1}}\left(\frac{h_L}{h_K}\right)^pd\tilde{C}_{q,K}$  as $F(L)$. Assume $L_j$ be a minimizing sequence such that $F(L_j)\rightarrow s$ as $j\rightarrow\infty$. Next, we need to prove that there exists a uniform constant $C>0$, such that for any $j$, $h_{L_j}\leq C$.
		
		We suppose to the contrary that, passing to a subsequence, $R_j:=\|h_{L_j}(e_j)\|_{L^{\infty}}\rightarrow\infty$ as $j\rightarrow\infty$. Since $K\in\mathcal{K}_e^n$, we have $h_K^{-p}d\tilde{C}_{q,K}=fd\mathcal{H}^{n-1}$ for some positive continuous function $f$. Therefore, 
		\begin{align*}
			F(L_j)=&\int_{\mathbb{S}^{n-1}}\left(\frac{h_{L_j}}{h_K}\right)^pd\tilde{C}_{q,K}
			=\int_{\mathbb{S}^{n-1}}h_{L_j}^pfd\mathcal{H}^{n-1}\\
			\geq&R_j^p\int_{\mathbb{S}^{n-1}}(x,e_j)_+^pfd\mathcal{H}^{n-1}
			\geq CR_j^p\rightarrow\infty
		\end{align*}
		as $j\rightarrow\infty$, which contradicts the fact that $F(L_j)\rightarrow s$ as $j\rightarrow\infty$. 
		
		By the Blaschke selection theorem, there exists a convex set $L_0\in\mathcal{K}^n$ and a subsequence $L_{j_k}$ such that $L_{j_k}\rightarrow L_0$ in Hausdorff distance. We next show that $L_0\in\kne$. Since $\tilde{V}_q(L_j)=1$, then the H\"older's inequality yields
		\begin{equation}\label{volume has lower bound}
			\begin{split}
				1=\tilde{V}_q(L_j)=&\frac{1}{n}\int_{\mathbb{S}^{n-1}}\rho_{L_j}^q(u)du\\
				\leq&\frac{1}{n}\left(\int_{\mathbb{S}^{n-1}}\rho_{L_j}^q(u)du\right)^{\frac{q}{n}}\left(\int_{\mathbb{S}^{n-1}}1du\right)^{1-\frac{q}{n}}\\
				=&n^{\frac{q}{n}-1}(\omega_{n-1})^{1-\frac{q}{n}}V_n(L_j)^{\frac{q}{n}}.
			\end{split}
		\end{equation}
		Hence, there exists a constant $C>0$ such that $V_n(L_j)\geq C$ for all $j$.
		Let $v_j$ be a direction where $h_{L_j}$ attains its minimum. The upper bound of $h_{L_j}$ shows that there exists $R>0$ such that $L_j\subseteq B_R$. Consider the symmetric strip 
		\begin{equation*}
			S_{L_j}:=\left\{x\in\mathbb{R}^n:|\langle x,v_j\rangle|\leq h_{L_j}(v_j)\right\}. 
		\end{equation*}
		It is obvious that $L_j\subseteq S_{L_j}\cap B_R$. We could deduce
		\begin{equation*}
			V_n(L_j)\leq V_n(S_{L_j}\cap B_R)=C_0R^{n-1}\min h_{L_j}
		\end{equation*}
		for some positive constant $C_0$. Combining this with $V_n(L_j)\geq C$, we obtain a uniform positive lower bound for $\min h_{L_j}$,
		\begin{equation*}
			\min h_{L_j}\geq \frac{C}{C_0R^{n-1}}>0.
		\end{equation*}
		Therefore, $L_0\in\kne$.
		Since $h_{L_{j_k}}\rightarrow h_{L_0}$ uniformly, we have $F(L_{j_k})\rightarrow F(L_0)$. Hence $F(L_0)=s$ and $L_0$ is the desired minimiser.   
	\end{proof}
	
	This lemma puts us in a position to prove Theorem \ref{thm 10}

	\begin{proof}[Proof of Theorem \ref{thm 10}]  By Lemma \ref{lem7.5}, $L_0$ is the minimiser.
		For any $f\in C^+_e(\mathbb{S}^{n-1})$, define
		\begin{equation*}
			q_t:=\log h_{L_0}(x)+tf(x)+o(t,x),~~~~~~~\text{for}~~~t\in(-\epsilon, \epsilon),
		\end{equation*}
		then $q_t\in C^+_e(\mathbb{S}^{n-1})$ for all $t\in(-\epsilon, \epsilon)$. Let $L_t$ be the Wulff shape associated with $q_t$, one has $L_t\in\mathcal{K}_e^n$ and $L_t$ is actually the minimiser $L_0$ when $t=0$.
		
		Since $h_{L_0}$ is the minimiser of $\Phi(\phi)$, then by the variational formula in \cite{HLYZ16}, we have
		\begin{align*}
			0=&\frac{d}{dt}\Big|_{t=0}\log\Phi(q_t)\\
			=&\frac{d}{dt}\Big|_{t=0}\log\left[\frac{1}{\tilde{V}_q([q_t])^{\frac{p}{q}}}\int_{\mathbb{S}^{n-1}}\left(\frac{h_{L_t}}{h_K}\right)^pd\tilde{C}_{q,K}\right]\\
			=&\frac{d}{dt}\Big|_{t=0}\left[-\frac{p}{q}\log\tilde{V}_q([q_t])+\log \left(\int_{\mathbb{S}^{n-1}}\left(\frac{h_{L_t}}{h_K}\right)^pd\tilde{C}_{q,K}\right)\right]\\
			=&\frac{p\int_{\mathbb{S}^{n-1}}\left(\frac{h_{L_0}}{h_K}\right)^{p-1}\frac{h_{L_0}}{h_K}fd\tilde{C}_{q,K}}{\int_{\mathbb{S}^{n-1}}\left(\frac{h_{L_0}}{h_K}\right)^pd\tilde{C}_{q,K}}-\frac{p\int_{\mathbb{S}^{n-1}} fd\tilde{C}_{q,L_0}}{\tilde{V}_q(L_0)}.
		\end{align*}
		From the arbitrariness of $f$, it follows that
		\begin{equation*}
			h_{K}^{-p}d\tilde{C}_{q,K}=\frac{\int_{\mathbb{S}^{n-1}}\left(\frac{h_{L_0}}{h_K}\right)^pd\tilde{C}_{q,K}}{\tilde{V}_q(L_0)}h_{L_0}^{-p}d\tilde{C}_{q,L_0}=h_{kL_0}^{-p}d\tilde{C}_{q,kL_0},
		\end{equation*}
		where the positive rescaling constant
		\begin{equation*}
			k=\left(\frac{\tilde{V}_q(L_0)}{\int_{\mathbb{S}^{n-1}}\left(\frac{h_{L_0}}{h_K}\right)^pd\tilde{C}_{q,K}}\right)^{\frac{1}{p-q}}.
		\end{equation*}
		
		Note that $K\in \mathcal{K}_{+,e}^{2,\alpha}$. By Caffarelli's regularity results \cite[Theorem 4]{C90}, we have $kL_0\in \mathcal{K}_{+,e}^{2,\alpha}$. Since $\|h_K-1\|_{C^{2,\alpha}}\leq\epsilon$, it means that
		\begin{equation*}
			\left\|\frac{h_K^{-p}d\tilde{C}_{q,K}}{d\mathcal{H}^{n-1}}-1\right\|_{C^\alpha(\sn)}<\epsilon_0
		\end{equation*}
		for some $\epsilon_0>0$. By the uniqueness result of $L_p$ dual Minkowski problem in \cite[Theorem 1.5]{BCLS-1}, it follows that $L_0=k^{-1}K\in \mathcal{K}_{+,e}^{2,\alpha}$.
		Therefore, for any 
		$L\in \mathcal{K}^n_e$, we have $h_L\in C^+_e(S^{n-1})$. Since $L_0=k^{-1}K$ is a minimizer of $\Phi$ and $\Phi$ is homogeneous of degree zero, we obtain
		\begin{equation*}
			\Phi(K)=\Phi(k^{-1}K)\leq\Phi(h_L),
		\end{equation*}
		and finally,
		\begin{equation*}
			\left(\int_{\sn}\left(\frac{h_L}{h_K}\right)^pd\bar{V}_{q,K}\right)^{\frac{1}{p}}\geq\left(\frac{\tilde{V}_q(L)}{\tilde{V}_q(K)}\right)^{\frac{1}{q}},
		\end{equation*}
		which is desired. 
	\end{proof}
	We finally prove the equivalence between the $L_p$
	Minkowski inequality and the $L_p$ Brunn–Minkowski inequality for the dual quermassintegral when $p,q>0$, and the method is from \cite{BLYZ12}. Notably, this equivalence holds without symmetry or smoothness assumptions.
	\begin{lem}\label{lem1}
		Suppose $p,q>0$. For $K,L\in\kno$, the $L_p$ Minkowski inequality \eqref{E:2} and the $L_p$ Brunn-Minkowski inequality \eqref{E:15} for the dual quermassintegral are equivalent.
	\end{lem}
	\begin{proof}
		Suppose $K,L\in\kno$. For $0\leq\lambda\leq1$, let
		\begin{equation*}
			Q_{\lambda}=(1-\lambda)K+_{p}\lambda L,
		\end{equation*}
		i.e., $Q_{\lambda}$ is the Wulff shape associated with the function $q_\lambda=((1-\lambda)h_K^p+\lambda h_L^p)^{\frac{1}{p}}$. It will be convenient to consider $h_{Q_{\lambda}}$ as being defined for $\lambda$ in the open interval $(-\varepsilon_0, 1+\varepsilon_0)$, where $\varepsilon_0>0$ is chosen so that for $\lambda\in(-\varepsilon_0, 1+\varepsilon_0)$, the function $q_\lambda$ is strictly positive.
		
		\textbf{Step 1: \eqref{E:2}$\Rightarrow$\eqref{E:15}.} 
		
		Assume that the $L_p$ Minkowski inequality \eqref{E:2} holds. Then
		\begin{align*}
			\tilde{V}_q(Q_{\lambda})
			=&\int_{\mathbb{S}^{n-1}}1d\tilde{C}_{q,Q_\lambda}\\
			=&\int_{\sn}h^{p}_{Q_\lambda}h^{-p}_{Q_\lambda}d\tilde{C}_{q,Q_\lambda}\\
			=&\int_{\mathbb{S}^{n-1}}((1-\lambda)h_K^p+\lambda h_L^p)h_{Q_{\lambda}}^{-p}d\tilde{C}_{q,Q_\lambda}\\
			=&(1-\lambda)\int_{\sn}\left(\frac{h_K}{h_{Q_\lambda}}\right)^pd\tilde{C}_{q,Q_\lambda}+\lambda\int_{\sn}\left(\frac{h_L}{h_{Q_\lambda}}\right)^pd\tilde{C}_{q,Q_\lambda}\\
			\geq&(1-\lambda)\tilde{V}_q(Q_{\lambda})\left(\frac{\tilde{V}_q(K)}{\tilde{V}_q(Q_{\lambda})}\right)^{\frac{p}{q}}+\lambda\tilde{V}_q(Q_{\lambda})\left(\frac{\tilde{V}_q(L)}{\tilde{V}_q(Q_{\lambda})}\right)^{\frac{p}{q}}\\
			=&\tilde{V}_q(Q_{\lambda})^{1-\frac{p}{q}}[(1-\lambda)\tilde{V}_q(K)^{\frac{p}{q}}+\lambda\tilde{V}_q(L)^{\frac{p}{q}}].
		\end{align*}
		Hence
		\begin{equation}\label{E:100}
			\tilde{V}_q(Q_{\lambda})\geq[(1-\lambda)\tilde{V}_q(K)^{\frac{p}{q}}+\lambda \tilde{V}_q(L)^{\frac{p}{q}}]^{\frac{q}{p}},
		\end{equation}
		which is actually \eqref{E:15}.
		
		\textbf{Step 2: \eqref{E:15}$\Rightarrow$\eqref{E:2}.}
		
		Define the function $f: [0,1]\rightarrow(0,\infty)$, given by $f(\lambda)=\tilde{V}_q(Q_{\lambda})^{\frac{p}{q}}$. \eqref{E:100} shows that $f(\lambda)\geq(1-\lambda)f(0)+\lambda f(1)$, which is not enough, because we want to verify that $f$ is a concave function, i.e., $f((1-\lambda)\sigma+\lambda \tau)\geq(1-\lambda)f(\sigma)+\lambda f(\tau)$ for $\sigma,\tau\in[0,1]$.
		When $p\geq1$, $f(\lambda)$ is concave. However, when $p<1$, it is not obvious, so a property of Wulff shapes is needed.
		For given $\sigma, \tau\in[0,1]$, let
		\begin{equation*}
			K_{\sigma}=(1-\sigma)\cdot K+_{p}\sigma\cdot L,\quad K_{\tau}=(1-\tau)\cdot K+_{p}\tau\cdot L.
		\end{equation*}
		Since $K_{\sigma}$ is the Wulff shape of the function $((1-\sigma)h_K^p+\sigma h_L^p)^{\frac{1}{p}}$, we have
		\begin{equation*}
			h_{K_{\sigma}}\leq((1-\sigma)h_K^p+\sigma h_L^p)^{\frac{1}{p}}.
		\end{equation*}
		If $\lambda\in[0,1]$ and $\alpha=(1-\lambda)\sigma+\lambda \tau$, this gives
		\begin{align*}
			(1-\lambda)h_{K_{\sigma}}^p+\lambda h_{K_{\tau}}^p\leq&(1-\lambda)[(1-\sigma)h_K^p+\sigma h_L^p]+\lambda[(1-\tau)h_K^p+\tau h_L^p]\\
			=&[(1-\lambda)(1-\sigma)+\lambda(1-\tau)]h_K^p+[(1-\lambda)\sigma+\lambda\tau]h_L^p\\
			=&(1-\alpha)h_K^p+\alpha h_L^p.
		\end{align*}
		Thus, $[(1-\lambda)h_{K_{\sigma}}^p+\lambda h_{K_{\tau}}^p]^{\frac{1}{p}}\leq[(1-\alpha)h_K^p+\alpha h_L^p]^{\frac{1}{p}}$, and taking the Wulff shapes of these functions allows us to conclude that
		\begin{equation*}
			(1-\lambda)K_{\sigma}+_{p}\lambda K_{\tau}\subseteq(1-\alpha)K+_{p}\alpha L.
		\end{equation*}
		This gives
		\begin{align*}
			f((1-\lambda)\sigma+\lambda \tau)=&\tilde{V}_q((1-\alpha)\cdot K+_{p}\alpha\cdot L)^{\frac{p}{q}}\\
			\geq&\tilde{V}_q((1-\lambda)\cdot K_{\sigma}+_{p}\lambda\cdot K_{\tau})^{\frac{p}{q}}\\
			\geq&(1-\lambda)\tilde{V}_q(K_{\sigma})^{\frac{p}{q}}+\lambda \tilde{V}_q(K_{\tau})^{\frac{p}{q}}\\
			=&(1-\lambda)f(\sigma)+\lambda f(\tau),
		\end{align*}
		which is the desired concavity of $f$, i.e., $f(1)-f(0)\leq f^{'}(0)$.
		
		From $f(\lambda)=\tilde{V}_q(Q_{\lambda})^{\frac{p}{q}}$, we have
		\begin{equation*}
			f^{'}(\lambda)|_{\lambda=0}=\frac{p}{q}\tilde{V}_q(Q_{\lambda})^{\frac{p}{q}-1}\tilde{V}_q(Q_{\lambda})^{'}|_{\lambda=0}.
		\end{equation*}
		Since $Q_{\lambda}$ is the Wulff shape of the function $q_{\lambda}^p=((1-\lambda)h_K^p+\lambda h_L^p)$, and
		\begin{equation*}
			q_{\lambda}^p=h_K^p+\lambda(h_L^p-h_K^p)=h_K^p\left[1+\lambda \left(\frac{h_L^p-h_K^p}{h_K^p}\right)\right],
		\end{equation*}
		it follows immediately that, for $\lambda\in[0,1]$,
		\begin{equation*}
			\log q_\lambda=\log h_K+\frac{1}{p}\log \left[1+\lambda \left(\frac{h_L^p-h_K^p}{h_K^p}\right)\right]=\log h_K+\frac{\lambda}{p}\left[\left(\frac{h_L}{h_K}\right)^p-1\right]+o(\lambda).
		\end{equation*}
		Therefore, combining the perturbation $g=\frac{1}{p}\left[\left(\frac{h_L}{h_K}\right)^p-1\right]$ with the variational formulation in \cite{HLYZ16}, we have
		\begin{align*}
			\frac{d}{d\lambda}\tilde{V}_q(Q_{\lambda})|_{\lambda=0}=&q\int_{\mathbb{S}^{n-1}}g(u)d\tilde{C}_{q,K}(u)\\
			=&\frac{q}{p}\int_{\mathbb{S}^{n-1}}\left[\left(\frac{h_L}{h_K}\right)^p-1\right]d\tilde{C}_{q,K}(u)\\
			=&\frac{q}{p}\int_{\mathbb{S}^{n-1}}\left(\frac{h_L}{h_K}\right)^pd\tilde{C}_{q,K}(u)-\frac{q}{p}\tilde{V}_q(K).
		\end{align*}
		Therefore, the concavity of $f$ yields
		\begin{align*}
			\tilde{V}_q(K)^{\frac{p-q}{q}}\left[\int_{\mathbb{S}^{n-1}}\left(\frac{h_L}{h_K}\right)^pd\tilde{C}_q(K,u)-\tilde{V}_q(K)\right]&=f^{'}(0)\\
			&\geq f(1)-f(0)=\tilde{V}_q(L)^{\frac{p}{q}}-\tilde{V}_q(K)^{\frac{p}{q}},
		\end{align*}
		which gives the $L_p$ dual Minkowski inequality.
	\end{proof}
	
	\begin{proof}[Proof of the second item in Theorem \ref{lp BMI no near ball}]
		Combining Theorem \ref{thm 10} and Lemma \ref{lem1}, we complete the proof. 
	\end{proof}
	
	\section*{Acknowledgement}
	The authors express sincere gratitude to professor Yong Huang for his constant guidance. The authors also thanks Dr. Jinrong Hu for helpful comments and suggestions.
	
	\section*{conflict of interest} On behalf of all authors, the corresponding author states that there is no conflict of interest.
	
	\section*{Data availability} There is no data used in this manuscript.
	
\end{document}